\documentclass[a4 paper, 10 pt]{article}
\usepackage{amsmath,amsthm,amssymb,amsfonts,amscd, graphicx}
\usepackage{amssymb, amsmath, amsthm}
\usepackage{graphicx}
\usepackage{tikz-cd}
\usepackage{caption}
\usepackage{subcaption}

\usepackage[all]{xy}
\usepackage[colorinlistoftodos]{todonotes}
\usepackage{float}
\usepackage{comment}
\usepackage{mathtools}
\DeclarePairedDelimiter\ceil{\lceil}{\rceil}
\DeclarePairedDelimiter\floor{\lfloor}{\rfloor}
\usepackage[left=1.1 in,top=1.2 in, right=1.1 in, bottom=1.3 in]{geometry}
\definecolor{cadmiumgreen}{rgb}{0.0, 0.42, 0.24}
\newtheorem{theorem}{Theorem}[section]
\newtheorem{lemma}[theorem]{Lemma}
\newtheorem{proposition}[theorem]{Proposition}

\theoremstyle{definition}

\newtheorem{remark}[theorem]{Remark}

\numberwithin{equation}{section}
\usepackage{lineno}

\usepackage{authblk}
\title{\bfseries{On the minimum spectral radius of connected graphs of given order and size}}
\author{Sebastian M. Cioab\u{a}}
\affil{University of Delaware}
\author{Vishal Gupta}
\affil{University of Delaware}
\author{Celso Marques}
\affil{CEFET-RJ }
\date{\today}

\begin{document}
\maketitle

\begin{center}
    {\em Dedicated to Professor Nair Maria Abreu}
\end{center}

\noindent
\begin{abstract}
In this paper, we study a question of Hong from 1993 related to the minimum spectral radii of the adjacency matrices of connected graphs of given order and size. Hong asked if it is true that among all connected graphs of given number of vertices $n$ and number of edges $e$, the graphs having minimum spectral radius (the minimizer graphs) must be almost regular, meaning that the difference between their maximum degree and their minimum degree is at most one. In this paper, we answer Hong's question positively for various values of $n$ and $e$ and in several cases, we determined the graphs with minimum spectral radius.
\end{abstract}

\tableofcontents

\section{Introduction}
Our graph notation is standard, see \cite{BH} for undefined terms or notations. The {\em eigenvalues} of a graph $G=(V,E)$ are the eigenvalues of its adjacency matrix $A=A(G)$. The largest eigenvalue of $G$, also called the \emph{spectral radius} or \emph{index} of the graph is related to the maximum degree, the average degree and other combinatorial parameters of the graph
and has been studied by many researchers in both theoretical and applied aspects (see \cite{CveRow2, SteBook, WCWF}). We denote the spectral radius of $G$ by $\rho(G)$ or by $\rho$ when the underlying graph is obvious. The maximum degree of $G$ is denoted by $\Delta(G)$ and its minimum degree by $\delta(G)$. For natural numbers $n$ and $e$, let $\mathcal{G}_{n,e}$ denote the set of all connected, undirected and simple graphs on $n$ vertices with $e$ edges. In 1993, Hong \cite{Hong} asked the following question: 
\begin{center}
{\em If $G\in \mathcal{G}_{n,e}$ has the smallest spectral radius among all graphs in $\mathcal{G}_{n,e}$, then is it true that $\Delta(G)-\delta(G)\leq 1$?}
\end{center}
We denote the smallest spectral radius among all graphs in $\mathcal{G}_{n,e}$ by $\rho_{min}(n,e)$ and we refer to any graph in $\mathcal{G}_{n,e}$ whose spectral radius equals $\rho_{min}(n,e)$ as a minimizer graph or minimizer of $\mathcal{G}_{n,e}$.

In 1957, Collatz and Sinogowitz \cite{CollSin} (cf. \cite{CveRow}) proved that the path $P_n$ is the unique minimizer graph among trees on $n$ vertices and therefore, is the minimizer of $\mathcal{G}_{n,n-1}$ (see also 
Lov\'asz and Pelik\'an \cite{LP}). By the Perron-Frobenius Theorem, this implies that $P_n$ is the unique minimizer graph among all simple connected graphs on $n$ vertices. In 1989, Simi\'{c} \cite{Simic2} proved that $B(k,n+1-2k, k)$ and $P(k, n+1-2k,k)$ (see Figure \ref{bicyclic minimizers}), where $k = \lceil\frac{n}{3}\rceil$, are the minimizers of $\mathcal{G}_{n,n}$. The graphs in $\mathcal{G}_{n,n}$ are sometimes referred to as unicyclic graphs while the graphs in $\mathcal{G}_{n,n+1}$ are called bicyclic graphs. Recently, in \cite{Stanic}, it was proved that for any natural number $s$, the complete bipartite graph $K_{s,s+1}$ is the unique minimizer graph in $\mathcal{G}_{2s+1,s(s+1)}$. We obtained this result independently around the same time (see Proposition \ref{e=s(s+1)}).
 \begin{figure}[h]
 \centering
\begin{tikzpicture}[scale = 0.6]
  \draw (4.8,0)--(6,0);
  \draw (0.8,-0.6)--(1.3,-1)--(1.8,-0.6)--(2,0)--(1.8,0.6)--(1.3,1)--(0.8,0.6);
  \draw (2,0)--(3.2,0);
  \draw[dotted](0.8,0.6) to[bend right =20] (0.8,-0.6);
   \draw (7.2,-0.6)--(6.7,-1)--(6.2,-0.6)--(6,0)--(6.2,0.6)--(6.7,1)--(7.2,0.6);
   \draw[dotted](7.2,0.6)to[bend left=20](7.2,-0.6);
  \draw[dotted](3,0)--(4.8,0);
  \foreach \x in {(1.3,1),(0.8,0.6),(1.8,0.6), (1.3,-1),(1.8,-0.6),(0.8,-0.6),(2,0),(3.2,0), (4.8,0),(6,0),(7.2,-0.6),(7.2,0.6),(6.7,-1),(6.7,1),(6.2,-0.6),(6.2,0.6)} {
  \draw[fill] \x circle[radius = 0.05cm];}
  \node[below] at (4,0) {\tiny q};
   \node[below] at (1.3,-1.1) {\tiny p};
    \node[below] at (6.7,-1.1) {\tiny r};
 
  \end{tikzpicture}
  \hspace{2cm}
  \begin{tikzpicture}[scale =0.6]
    \draw[dotted](-1,0)--(1,0);
    \draw (-1.2,-1)--(-2,0)--(-1.2,1);
    \draw[dotted] (-1.2,1) to[bend left=20](1.2,1);
    \draw (-2,0)--(-1,0);
    \draw[dotted](-1.2,-1) to[bend right=20] (1.2,-1);
    \draw (1.2,-1)--(2,0)--(1.2,1);
    \draw (1,0)--(2,0);
      \foreach \x in {(-2,0),(-1,0),(2,0),(1,0),(-1.2,1),(-1.2,-1),(1.2,-1),(1.2,1)}
      {\draw[fill] \x circle[radius = 0.05cm];}
      \node[below] at (0,0) {\tiny q};
      \node[below] at (0,1.1) {\tiny p};
       \node[below] at (0,-1.2) {\tiny r};
  \end{tikzpicture}
  \caption{$B(p,q,r)$ (left) and $P(p,q,r)$ (right).}
  \label{bicyclic minimizers}
  \end{figure}
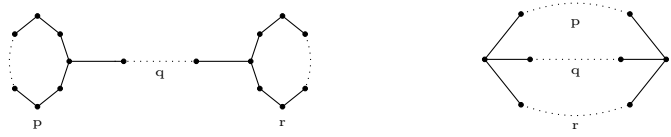
  
In this paper, we answer Hong's question affirmatively for dense graphs when $e\geq {n-1\choose 2}-2$ (see Propositions \ref{n-2 to n-1}, \ref{e=n-1 choose 2}, \ref{p=n+1/2}, \ref{p=n+2/2}, \ref{n-2 to n-3}, \ref{n-2 to n-3 2} and \ref{n-3 regular -1}) and for some sporadic cases when $e = \frac{n^2}{4}-1$ (via Proposition \ref{min for e=n^2/4-1}) or $e=\frac{n^2}{3}-1$ (see Proposition \ref{minimizer for e=n^2/3-1}). In all these cases, we determined the minimizer graphs. We end with a discussion on some observations and future directions. 

The corresponding maximizing problem, namely finding the maximum spectral radius and the structural properties of graphs attaining the maximum spectral radius among the graphs in $\mathcal{G}_{n,e}$, has been studied by several researchers (see \cite{BLS, Amit, Brualdihoffman, BruSol, CveRow} for example).

\section{Preliminaries}

In this section, we introduce the notations and some results from the literature that will be used in the subsequent sections. Let $G=(V,E)$ be a graph. For $u,v\in V$, we denote their adjacency and non adjacency in $G$ by $u\sim v$ and $u\not\sim v$, respectively. For a vertex $v\in V$, we denote its degree by $d(v)$ and the set of all vertices adjacent to $v$ by $N(v)$. Let $d = (d_1,\ldots, d_n)$ be the degree sequence of $G$. For $1\leq p\leq \infty$, the $p$-mean of $d$ is defined as $ d^{(p)}= (\frac{1}{n}\sum_{i=1}^n d_i^p)^{1/p}.$ In \cite{Hof} the author defined \emph{the spectral mean characteristic} of $G$, denoted by $char_{\rho}(G)$, to be the unique constant $1\leq p\leq \infty$ such that $\rho$ equals the $p$-mean of its degree  sequence $d$.  If $G$ is regular, then $char_{\rho}(G) =1$.

\begin{lemma}[\cite{Hof}]\label{increasing p}
The $p$-mean $d^{(p)}$, considered as a function of $p$, is continuous and strictly increasing if and only if there exists $i,j\in [n]$ such that $d_i\not=d_j.$  
\end{lemma}
\begin{theorem}[\cite{Hof}]\label{charG}
For an irregular graph $G$, $char_{\rho}(G)\geq 2$.  
\end{theorem}
\begin{theorem}[\cite{Hof}]\label{semiregular}
 If $G$ is semiregular, then $char_{\rho}(G)=2.$ If $G$ is connected and $char_{\rho}(G) =2,$ then $G$ is semiregular.   
\end{theorem}

If $G$ is a connected graph with adjacency matrix $A$ and spectral radius $\rho(G)$, the Perron-Frobenius implies that the eigenspace corresponding to $\rho(G)$ is $1$-dimensional and is spanned by an eigenvector with positive entries. The principal eigenvector of $G$ is the eigenvector corresponding to $\rho(G)$ that has all entries positive and has length one. 

\begin{theorem}[Theorem 8.1.3, \cite{CRS}]\label{rotation}
Let $G=(V,E)$ be a connected graph with spectral radius $\rho(G)$ and principal eigenvector $x$. Assume that $r,s$, and $t$ are vertices of $G$ such that $r\sim s$, $r\not\sim t$ and $x_t\geq x_s$. If $G'$ is the graph obtained from $G$ by deleting the edge $rs$ and adding an edge between $r$ and $t$, then $\rho(G')>\rho(G)$. 
\end{theorem}

\begin{theorem}[Theorem 8.1.10, \cite{CRS}]\label{local switching}
Let $G=(V,E)$ be a connected graph with spectral radius $\rho(G)$ and corresponding principal eigenvector $x$. Assume that $r, s, u$, and $v$ are vertices of $G$ such that $s\sim t,u\sim v, s\not\sim v, t\not\sim u$ and $(x_s-x_u)(x_v-x_t)\geq 0$. Let $G'$ be the graph obtained from $G$ by replacing the edges $st$ and $uv$ by $sv$ and $tu$, then $\rho(G')\geq \rho(G)$. Equality happens if and only if $x_s = x_u$ and $x_v = x_t$.   
\end{theorem}

Given a graph $G=(V,E)$ and a natural number $k$, a partition $\pi$ of its vertex $V=X_1\cup \ldots X_k$ into $k$ parts is called equitable if there exists non-negative integers $b_{i,j}$ for $1\leq i,j\leq k$, such that for any $i,j\in \{1,\ldots,k\}$ and any vertex $u\in X_i$, the number of neighbors of $u$ that are contained in $X_j$ equals $b_{i,j}$. The $k\times k$ matrix $B=(b_{i,j})_{1\leq i,j\leq k}$ is called the quotient matrix of the equitable partition $\pi$. The following result is well-known (see \cite[Section 2.3]{BH} or \cite[Section 9.3]{GR}).
\begin{theorem}
If $\pi$ is an equitable partition of a graph $G=(V,E)$ with quotient matrix $B$, then any eigenvalue of $B$ is an eigenvalue of $A$. 
\end{theorem}
We will use the following consequence of the result above (which follows using the Perron-Frobenius theorem, see \cite[p.214, Ex. 8] {GR}).
\begin{theorem}\label{thm:eq_sr}
Let $G$ be a connected graph. If $\pi$ is an equitable partition of $G$ with quotient matrix $B$, then the spectral radius of $G$ equals the spectral radius of $B$.
\end{theorem}

Another tool that we will use in our paper is eigenvalue interlacing. For a graph $G=(V,E)$, an natural number $k$ and a partition of $V$ into $k$ parts: $V=Y_1\cup \ldots \cup Y_k$, define the quotient matrix $Q$ of this partition to be $k\times k$ matrix whose $(i,j)$-th entry $q_{i,j}$ equals the average number of neighbors contained in $Y_j$ of the vertices in $Y_i$. The following result is also well-known (see \cite[Section 2.5]{BH} or \cite[Lemma 9.5]{GR} for more details).
\begin{theorem}\label{thm:int_sr}
Let $G$ be a graph. If $\pi$ is a partition with quotient matrix $Q$, then the eigenvalues of $Q$ interlace the eigenvalues of $G$ and therefore, $\rho(G)\geq \rho(Q)$. If the interlacing is tight, then the partition $\pi$ is equitable.
\end{theorem}

\section{Dense graphs}

We start this section with two simple observations. 
\begin{proposition}\label{reg}
If $G$ is a $(n,e)$-graph, then $\rho(G)\geq \frac{2e}{n}$. Equality happens if and only if $n|2e$ and $G$ is a $\frac{2e}{n}$-regular graph.
\end{proposition}
\begin{proof}
If $\mathbf{1}$ denotes the all-ones column vector of dimension $n$, then $\rho(G)\geq \frac{\mathbf{1}^TA(G)\mathbf{1}}{\mathbf{1}^T\mathbf{1}} = \frac{2e}{n}.$ Equality holds if and only if $\mathbf{1}$ is an eigenvector for $\rho$, that is, if and only if, $G$ is a $\frac{2e}{n}$-regular graph.
\end{proof}
Note that for any natural numbers $n$ and $e$ with $n|2e$, there exist a $\frac{2e}{n}$-regular graph on $n$ vertices. This observation and the previous Proposition show that when $2e$ is not divisible by $n$, Hong's question concerns irregular graphs. There are several measures for graph irregularity in literature and here, we use the irregularity $Ir(G)$ of a graph $G$ which is defined by $Ir(G) = \Delta(G) - \delta(G)$. We call a graph $G$ \emph{almost regular} if $Ir(G) =1$. Hence, one can rephrase Hong's question as : 
\begin{center}
{\em Is it true that the minimizers in $\mathcal{G}_{n,e}$ are almost regular graphs?}
\end{center}
The next lemma shows that except for the trees, no minimizer can have vertices of degree one.
\begin{lemma}\label{degree1}  
If $e\geq n\geq 2$, then a minimizer graph in $\mathcal{G}_{n,e}$ does not contain a vertex of degree one.
\end{lemma}
\begin{proof}
Let $G=(V,E)$ be a minimizer graph in $\mathcal{G}_{n,e}$. Suppose for contradiction that there is a vertex $v$ of degree one in $G$. 

An internal path of a graph $H$ is defined as a sequence of distinct vertices $u_1,\ldots, u_k$ (except possibly $u_1=u_k$) such that $\min(d(u_1),d(u_k))\geq 3, d(u_2)=\ldots =d(u_{k-1})=2$ (unless $k=2$), and $u_j$ is adjacent to $u_{j+1}$ for $j\in \{1,\ldots,k-1\}$. Two adjacent vertices, each of degree three or more, form an internal path. A cycle in which $u_1=u_k$ has degree $3$ or more is also an internal path. If $vw$ is an edge of $H$, denote by $H_{v,w}$ the graph obtained from $H$ by subdividing the edge $vw$ by one vertex $z$; equivalently, $H_{v,w}$ is obtained from $H$ by adding one new vertex $z$, by deleting the edge $uv$ and adding the edges $vz$ and $zw$. Hoffman and Smith \cite{HS} (see also \cite{McKee} for a generalization) proved that if $H$ is a graph and $vw$ is an edge of an internal path of $H$, then $\rho(H_{v,w})\leq \rho(H)$.

Because $e\geq n\geq 2$, the graph $G$ must contain an internal path $P$. Since $d(v)=1$, $v$ is not contained in $P$. Let $ab$ be an edge of $P$. Hence, $\rho(G_{a,b})\leq \rho(G)$ by the Hoffman-Smith result mentioned above. Note that $G_{a,b}$ is a connected graph with $n+1$ vertices and $e+1$ edges. If $G'$ denotes the graph obtained from $G_{a,b}$ by deleting the vertex $v$, then $G'$ is connected, has $n$ vertices and $e$ edges and $\rho(G')<\rho(G_{a,b})$. Therefore, $G'\in \mathcal{G}_{n,e}$ and $\rho(G')<\rho(G)$, contradiction with $G$ being a minimizer in $\mathcal{G}_{n,e}$. This finishes our proof. \end{proof}
%

\subsection{The case $e\geq \frac{n(n-2)}{2}$}

Let $n$ be a natural number. We denote by $K_n$ the complete graph on $n$ vertices. If $n$ is even, the cocktail party graph $CP_n$, also known as the hyperoctahedral graph, is the complement of a perfect matching of $K_n$. Given two vertex disjoint graphs $G=(V,E)$ and $H=(W,F)$, the join $G\vee H$ of $G$ and $H$ is the graph with vertex set $V\cup W$ and whose edge set is $E\cup F\cup \{vw: v\in V, w\in W\}$.
\begin{proposition}\label{n-2 to n-1}
Let $n\in \mathbb{N}$, $1\leq p \leq \lfloor \frac{n}{2}\rfloor$ and $e = {n\choose 2} - p$. Then 
\begin{equation}
\rho_{min}(n,e)= \frac{n-3 + \sqrt{(n+1)^2 -8p}}{2}.
\end{equation}
If $n$ is even and  $p= \frac{n}{2}$, then $CP_{2p}$ is the minimizer of $\mathcal{G}_{n,\binom{n}{2}-p}$. If $p<\frac{n}{2}$, then $K_{n-2p}\vee CP_{2p}$ is the minimizer of $\mathcal{G}_{n,\binom{n}{2}-p}$.\end{proposition}
\begin{proof}
 Let $G$ be a graph in $\mathcal{G}_{n,e}$. If $n$ is even and $p = \frac{n}{2}$, the result follows from Proposition \ref{reg}. Assume that $1\leq p< \frac{n}{2}.$ Note that $G$ is obtained from $K_n$ by removing $p$ edges. Thus, $G$ has at least $n-2p$ vertices of degree $n-1$ and $G = K_{n-2p}\vee H$, where $H$ is a graph of order $2p$ and size $2p(p-1).$ Consider the partition of vertex set of $G$ into two parts: the vertex set of $K_{n-2p}$ and the vertex set of $H$. The quotient matrix of this partition is 
    \begin{equation*}
    Q = \begin{bmatrix}
        n-2p-1&2p\\
        n-2p&2p-2
    \end{bmatrix}.
    \end{equation*}
The largest eigenvalue $\rho(Q)$ of $Q$ is the largest root of the polynomial
\begin{equation*}
    P_Q(x) = x^2 - (n-3)x + 2p+2-2n,
\end{equation*}
and equals
\begin{equation*}
\rho(Q) = \frac{n-3 +\sqrt{(n+1)^2-8p}}{2}.
\end{equation*}
By eigenvalue interlacing (see Theorem \ref{thm:int_sr}), it follows that $\rho(G)\geq \rho(Q)=\frac{n-3 + \sqrt{(n+1)^2 -8p}}{2}$ and equality happens if and only if this partition is equitable which happens exactly when $H$ is a regular graph. Because $H$ has $2p$ vertices and $2p(p-1)$ edges, $H$ must be $CP_{2p}$. This finishes our proof. \end{proof}

\subsection{The case ${n-1\choose 2} \leq e \leq \frac{n(n-2)}{2}$ }
For $n,r\in \mathbb{N}$,  we denote by $G_n^r$ a $(n-r)$-regular graph on $n$ vertices. When $e=\binom{n}{2}-(n-1)=\binom{n-1}{2}$, Jack Koolen\footnote{Personal communication to the first author during the conference Workshop on Spectral Graph Theory, Niteroi, Brazil {\tt http://spectralgraphtheory.org/} in October 2023.}  asked whether the join of $G_{n-2}^{3}$ (or the complement of a $2$-regular graph on $n-2$ vertices) and two isolated vertices is a minimizer graph. We show that this is correct. First, we calculate the spectral radius of such graphs. 
\begin{proposition}
If $H$ is the join of the complement of a $2$-regular graph on $n-2$ vertices and two isolated vertices, then $\rho(H)$ is the largest root of the polynomial 
    \begin{equation}
        P_H(x)=x^2-(n-5)x-2(n-2),
    \end{equation}
    and consequently,
    \begin{equation}
    \rho(H)=\frac{n-5+\sqrt{(n-1)^2+8}}{2}>n-3.
    \end{equation}
\end{proposition}
\begin{proof}
Consider the partition of the vertex set of $H$ into two parts: the two isolated vertices and the rest. This partition is equitable with quotient matrix
    \begin{equation}
    B=\begin{bmatrix}0&n-2\\2&n-5\end{bmatrix}.
    \end{equation}
Theorem \ref{thm:eq_sr} implies that $\rho(H)$ equals the largest eigenvalue of $B$. A simple calculation finishes the proof.
\end{proof}

\begin{proposition}\label{e=n-1 choose 2}
 If $G$ is a graph on $n$ vertices and $\binom{n-1}{2}$ edges, then 
    \begin{equation}
        \rho(G)\geq \frac{n-5+\sqrt{(n-1)^2+8}}{2}.
    \end{equation}
\end{proposition}
\begin{proof}
Let $G=(V,E)$ be a graph with $n$ vertices and $\binom{n-1}{2}$ edges. We denote by $\Delta(G)$ the maximum degree of $G$. If $\Delta(G)=n-1$, then let $u$ be a vertex of $G$ of degree $n-1$. Consider the partition of the vertex set of $G$ into two parts: $\{u\}$ and $V\setminus \{u\}$. Because the number of edges of $G$ is $\binom{n-1}{2}$, the quotient matrix of this partition is
\begin{equation}
Q_G=\begin{bmatrix}0 & n-1\\
1 & n-4
\end{bmatrix}
\end{equation}
The largest eigenvalue $\rho(Q_G)$ of $Q_G$ is the largest root of the polynomial 
\begin{equation}
    P_{Q_G}(x)=x^2-(n-4)x-(n-1),
\end{equation}
and equals 
\begin{equation}
\rho(Q_G)=\frac{n-4+\sqrt{(n-2)^2+8}}{2}.
\end{equation}
By eigenvalue interlacing (see Theorem \ref{thm:int_sr}), we get that 
\begin{equation}\label{eq:interlaceG}
\rho(G)\geq \rho(Q_G).
\end{equation}
We now show that 
\begin{equation}
    \rho(Q_G)>\rho(H).
\end{equation}
To see this, first note (from the previous proposition) that $\rho(H)^2=(n-5)\rho(H)+2(n-2)$ and 
that 
\begin{equation}
\rho(H)=\frac{n-5+\sqrt{(n-1)^2+8}}{2}>\frac{n-5+n-1}{2}=n-3.
\end{equation}
Therefore,
\begin{align*}
    P_{Q_G}(\rho(H))&=\rho(H)^2-(n-4)\rho(H)-(n-1)\\
    &=(n-5)\rho(H)+2(n-2)-(n-4)\rho(H)-(n-1)\\
    &=(n-3)-\rho(H)\\
    &<0.
\end{align*}
This inequality implies that $\rho(H)$ lies between the roots of the polynomial $P_{Q_G}$ and therefore, $\rho(Q_G)>\rho(H)$. Using \eqref{eq:interlaceG}, we deduce that $\rho(G)>\rho(H)$. This finishes the proof of this case.

If $\Delta(G)\leq n-2$, then we first show that $\Delta(G)=n-2$. Because $e=\binom{n-1}{2}$, we have that $\Delta(G)\geq \frac{2e}{n}=\frac{n^2-3n+2}{n}=n-3+\frac{2}{n}$. Therefore, $\Delta(G)\geq n-2$ which proves our claim.

Let $t$ denote the number of vertices $v$ of $G$ whose degree equals $n-2$. Thus,
\begin{equation}
n^2-3n+2=2e=\sum_{u\in V}d(u)\leq t(n-2)+(n-t)(n-3)=n^2-3n+t.
\end{equation}
Therefore, $t\geq 2$.

If $t=2$, then let $v_1$ and $v_2$ be the two vertices of degree $n-2$. From the above inequality, we deduce that the remaining vertices have degrees equal to $n-3$.  There are two sub-cases here: if $v_1$ and $v_2$ are adjacent and if $v_1$ and $v_2$ are not adjacent.

If $v_1$ and $v_2$ are not adjacent, then the graph $G$ is the join of an empty graph on two vertices ($v_1$ and $v_2$) and the complement of a $2$-regular graph on $n-2$ vertices and falls into the category of the conjectured minimizer graphs.

If $v_1$ and $v_2$ are adjacent, then consider the partition of $G$ into $\{v_1,v_2\}$
 and $V\setminus \{v_1,v_2\}$. The quotient matrix of this partition is 
\begin{equation}
Q'=\begin{bmatrix}
1 & n-3\\
\frac{2n-6}{n-2} & n-5+\frac{2}{n-2}
\end{bmatrix}.
\end{equation}
The largest eigenvalue $\rho(Q')$ of $Q'$ is the largest root of the polynomial 
\begin{equation}
P_{Q'}(x) = x^2 -x\left(n-4 + \frac{2}{n-2}\right) + (n-5) + \frac{2}{n-2}(1-(n-3)^2). 
\end{equation}
We compute
\begin{equation}
 (P_H - P_{Q'})(x) = x\frac{n}{n-2} -3(n-3) +\frac{2}{n-2}((n-3)^2-1).   
\end{equation}
When $x = \rho(H)$, the above equation becomes
\begin{align*}
-P_{Q'}(\rho(H))&= \frac{n(n-5)}{2(n-2)} + \frac{n\sqrt{(n-1)^2 +8}}{2(n-2)} -3(n-3) + \frac{2(n-2)(n-4)}{n-2}\\
&= \frac{1}{2(n-2)}\left( n(n-5) + n\sqrt{(n-1)^2 +8} + (n-2)(-2n+2)\right)\\
&=\frac{-1}{2(n-2)}\left(n(n-1) +4 - n\sqrt{(n-1)^2 +8}\right).
\end{align*}
The right hand side is positive if and only if $n(n-1) +4 < n\sqrt{(n-1)^2 +8}$ which is true for all $n>2.$ Therefore when $x = \rho(H)$, $P_{Q'}(x)<0$. This proves that $\rho(Q')>\rho(H)$ and 
by eigenvalue interlacing (see Theorem \ref{thm:int_sr}), we get that $\rho(G)>\rho(H)$.

If $t\geq 3,$ then there will be two vertices $v_1$ and $v_2$ (each of degree $n-2$) that are adjacent in $G$. We have shown above that such graph cannot be a minimizer. This completes the proof.
\end{proof}

\begin{lemma}\label{not adjacent n-2}
A minimizer graph $G\in \mathcal{G}_{n,e}$ cannot contain a vertex $v$ of degree $n-2$ and a vertex $u$ of degree $n-i$, for some $i\geq 4$, such that $v\not \sim u$ in $G$. 
\end{lemma}
\begin{proof}
Suppose a minimizer $G\in \mathcal{G}_{n,e}$ has a vertex $v$ with $d(v)=n-2$ that is not adjacent to a vertex $u$ with $d(u)=n-i$ for some $i\geq 4$. Choose a vertex $w$ such that $w\not\sim u$ in $G$. We construct a graph $H$ by deleting the edge $vw$ and adding the edge $uw.$ Let $x$ be the principal eigenvector of $H.$ If $x_u\leq x_v,$ then $\lambda_1(H)<\lambda_1(G)$ by Theorem \ref{rotation}, a contradiction. If $x_u>x_v,$ then we add edges $uz$ and delete edges $vz$ for all $z\not\in N(u)$ to obtain a graph $H'$ which is isomorphic to $G$. Then again $\lambda_1(H)<\lambda_1(H')=\lambda_1(G),$ a contradiction. 
\end{proof}

\begin{theorem}\label{n-1}
Let $n\geq 3$ be an integer. If $e<{n\choose 2} -\frac{n}{2}$, then a minimizer graph in $\mathcal{G}_{n,e}$ cannot have a vertex of degree $n-1$.   
\end{theorem}
\begin{proof} Let $G$ be a minimizer graph of $\mathcal{G}_{n,e}$. Assume that $G$ has a vertex $v$ with degree $n-1$. Choose a vertex $u$ such that $d(u)<\floor{2e/n}<n-2$, and a vertex $w\not\sim u$ in $G$. We construct a new graph $G'$ by deleting the edge $vw$ and adding the edge $uw$ in $G$. Note that $d(u)<n-1$ in $G'$. Let $x$ be the principal eigenvector of $G'$. If $x_u\leq x_v$, then $\lambda_1(G')<\lambda_1(G)$ by Theorem \ref{rotation}, a contradiction. If $x_u>x_v$, we add the edges $uz$ and delete the edges $vz$ for all $z\not\in N(u)$ in $G'$ to obtain the graph $H$, which is isomorphic to $G$. Then clearly $\lambda_1(G')<\lambda_1(H)=\lambda_1(G)$, a contradiction. This completes the proof.
\end{proof}

If $G=(V,E)$ and $H=(W,F)$ are two vertex disjoint graphs, then $G\cup H$ denotes the graph with vertex set $V\cup W$ and edge set $E\cup F$.
\begin{proposition}\label{p=n+1/2}
If $n\geq 5$ is odd and $e = {n\choose 2} - \frac{n+1}{2}$, then $\rho_{min}(n,e)$ equals  the largest eigenvalue of the matrix
\begin{equation*}
\begin{bmatrix}
            n-5&2&1\\
            n-3&1&0\\
            n-3&0&0
        \end{bmatrix}.
\end{equation*}
The minimizers in $\mathcal{G}_{n,e}$ are of the form $G_{n-3}^2\vee (K_2 \cup K_1)$.
\end{proposition}
\begin{proof}
Let $G=(V,E)$ be a minimizer in $\mathcal{G}_{n,e}$. Denote by $\delta(G)$ the minimum degree of $G$. By Theorem \ref{n-1}, $G$ cannot have a vertex of degree $n-1$ and so $\Delta(G)\leq n-2$. Suppose that $\delta(G)\leq n-4$. Assume that $G$ has exactly $t$ vertices of degree $n-2$. Thus, $n(n-2) -1=2e= \sum_{v\in V}d_v \leq  n-4 + t(n-2) + (n-t-1)(n-3)$. After solving for $t$, we get $t \geq n$, which is not possible and so $\delta(G)\geq n-3$. Since $G$ is not a regular graph, $\Delta(G) -\delta(G)=1$. Also, $n-1$ vertices have degree $n-2$ and the remaining vertex (call it $u$) has degree $n-3.$ Let $v,w\in V$ such that $v,w\not\in N(u)$. Because $d(v)=d(w)=n-2$, we get $G =H\vee (K_2 \cup K_1)$, where $H$ is a $(n-5)$-regular graph on $n-3$ vertices. This proves the claim. 
\end{proof}
\begin{proposition}\label{p=n+2/2}
If $n\geq 6$  is even and $e = {n\choose 2} - \frac{n+2}{2}$, then $\rho_{min}(n,e)$ equals the largest eigenvalue of the matrix
        $$\begin{bmatrix}
            n-6&2&2\\
            n-4&1&1\\
            n-4&1&0
        \end{bmatrix}.$$
The minimizers in $\mathcal{G}_{n,e}$ are of the form $G_{n-4}^2\vee P_4$.   
\end{proposition}
\begin{proof}
Let $G' = (V',E')$ be a minimizer in $\mathcal{G}_{n,e}$. By Theorem \ref{n-1}, $\Delta(G')\leq n-2.$ Assume that there are exactly $t$ vertices of degree $n-2$ in $G'$. If $\delta(G')\leq n-4$, then $2e = n(n-2) -2 = \sum_{v\in V'}d_v \leq  n-4 + t(n-2) + (n-t-1)(n-3).$ After solving for $t$, we get $t \geq n-1.$ Therefore, the only possible case is when one vertex has degree $n-4$ and the remaining vertices have degree $n-2$, meaning that $G' = G_{n-4}^2\vee (K_3\cup K_1)$. Since $G'$ has two non-adjacent vertices, one of degree $n-2$ and the other of degree $n-4$, Lemma \ref{not adjacent n-2} implies that $G'$ cannot be a minimizer. This proves $\delta(G')\geq n-3$. As $G'$ has $t$ vertices of degree $n-2$ and $n-t$ vertices of degree $n-3$, we have that $2e = n(n-2)-2 = t(n-2) + (n-t)(n-3)$. Solving for $t$ we get $t = n-2.$ Hence in $G'$ we have $n-2$ vertices of degree $n-2$ and two vertices (say $u,v)$ of degree $n-3.$

If $u\not\sim v$, then $G' \cong G_{n-4}^2\vee P_4,$ and we are done. If $u\sim v$, then $G'\cong G^2_{n-6}\vee(K_4\cup P_2)$. Consider the partition of $G'$ into three parts: $G_{n-6}^2,$ $K_4,$ and $P_2.$ The partition is equitable with the quotient matrix
     $$Q_G' = \begin{bmatrix}
         n-8&4&2\\
         n-6&3&1\\
         n-6&2&1
     \end{bmatrix}.$$
Theorem \ref{thm:eq_sr} implies that $\rho(G')$ is the largest root of the characteristic polynomial $P_{Q_G'}(x) = x^3-x^2(n-4)-x(2n-5)+n-4.$ 
     Now consider the partition of $G = G_{n-4}^2\vee P_4$ into three parts: $G_{n-4}^2$, the inner two vertices of $P_4$, and the outer two vertices of $P_4.$ This partition is equitable with quotient matrix
     $$Q_G =\begin{bmatrix}
         n-6&2&2\\
         n-4&1&1\\
         n-4&1&0
     \end{bmatrix}.$$
  Theorem \ref{thm:eq_sr} implies that $\rho(G)$ is the largest root of the characteristic polynomial $P_{Q_G}(x) = x^3 -x^2(n-5)-3x(n-3)-n+2.$ 
     We note that $(P_{Q_G}-P_{Q_G'})(x) = x^2-x(n-4)-2n+6$ is monotonically increasing for $x>\frac{n-4}{2}.$ Because $\min(\rho(G),\rho(G')) > n-2-\frac{2}{n}$ and  
     $ (P_{Q_G}-P_{Q_G'})\left(n-2-\frac{2}{n}\right) =\frac{4}{n^2}>0,$ we get that $\rho(G)<\rho(G').$ This completes the proof.
\end{proof}

\begin{theorem}\label{n-2 to n-3}
Let $n\geq 6$ be an even integer. If $1\leq p\leq \frac{n-4}{2}$ and $e = {n\choose 2} - (\frac{n+2}{2}+p)$, then 
\begin{equation*}
    \rho_{min}(n,e) = \frac{n-5+\sqrt{(n+1)^2-8p-8}}{2}.
\end{equation*}
The minimizers of $\mathcal{G}_{n,e}$ are of the form $G_{n-2p-2}^2\vee G_{2p+2}^3$.
\end{theorem}
\begin{proof}[Proof of Theorem \ref{n-2 to n-3}] 
We will need the next lemma where we use the following graph operation known as the Kelmans transformation. Consider two vertices $u$ and $v$ of a graph $G$. Construct a new graph $G'$ from $G$ by deleting the edge $vw$ and adding the new edge $uw$ for each vertex $w$ that is adjacent to $v$ and is not adjacent to $u$. The vertices $u$ and $v$ are adjacent in $G'$ if and only if they are adjacent in $G$. Kelmans \cite{Kelman} introduced this operation in the context of some problems in random graph theory and later,  Csikv\'{a}ri \cite{Csikvari} proved that the Kelmans transformation increases the spectral radius of the graph, namely that $\rho(G')\geq \rho(G)$.
\begin{lemma}\label{Kelmans1}
 Let $G\in\mathcal{G}_{n,e}$ and let $u,v\in V(G)$ with $d(u)=d$ and $d(v) =d+l$. If $l\geq 2$ and $N(u)\setminus\{v\}\subset N(v)$, then $G$ is not a minimizer in $\mathcal{G}_{n,e}.$
\end{lemma}
\begin{proof}
Suppose  $l\geq 2$ and $N(u)\setminus\{v\}\subset N(v)$. Pick a vertex $w\in N(v)\setminus N(u)$ such that $w\neq u$. Delete the edge $vw$ and add the edge $uw$ to obtain a new graph $H$. We note that $H\in \mathcal{G}_{n,e}$ and the graph $G$ can be obtained from $H$ by doing the Kelmans transformation on the vertices $v$ and $u$. Therefore, $\rho(G)>\rho(H)$. This shows that $G$ is not a minimizer in $\mathcal{G}_{n,e}$.  
\end{proof}

The following two claims prove Theorem \ref{n-2 to n-3}.

{\bf Claim 1} Any graph (other than the claimed minimizers of the form $G_{n-2p-2}^2\vee G_{2p+2}^3$ with some vertices of degree $n-2$ inducing $G_1=G^2_{n-2p-2}$ as its subgraph cannot be a minimizer.
    
{\bf Claim 2} In a minimizer, any subgraph induced by $n-2p-2$ vertices of degree $n-2$ is in join with the rest of the graph.


 {\em Proof of Claim 1} Let $G = G_{n-2p-2}^2\vee G_{2p+2}^3$, $G_1 = G^2_{n-2p-2}$, $G_2=G^3_{2p+2}$, $V_1=V(G_1),$ and $V_2 = V(G_2)$. Let $V(G)=V_1\cup V_2$ in this order. The partition of $V$ into $V_1$ and $V_2$ is equitable and the corresponding quotient matrix is 
 \begin{equation}\label{eq:qgthm310}
 Q_G = \begin{bmatrix}
        n-2p-4&2p+2\\
        n-2p-2&2p-1
    \end{bmatrix}.
\end{equation}
    
Let $x =(x_1,x_2)^T$ be the principal eigenvector of $Q_G$. Because the partition $V=V_1\cup V_2$ is equitable, Theorem \ref{thm:eq_sr} implies that
    \begin{equation}\label{eq:rhoqg310}
    \rho(G)=\rho(Q_G)=\frac{n-5+\sqrt{(n+1)^2-8p-8}}{2}.
    \end{equation}    
Define the vector $\vec{z}=(z_1, z_2, ..., z_n)^T$, where $z_i = x_1$ for $1\leq i\leq |V_1|$ and $z_i = x_2$ for $|V_1|+1\leq i\leq n$, and denote $Y=\frac{\vec{z}}{||\vec{z}||_2}$. Theorem \ref{thm:eq_sr} also implies that $\vec{z}$ and $Y$ are eigenvectors of $A(G)$ corresponding to $\rho(G)$ with $Y$ being the principal eigenvector.
   
We use proof by contradiction for the remaining part of the proof. Let $H=(V(H),E(H))$ be a minimizer in $\mathcal{G}_{n,e}$ that is not of the form $G_{n-2p-2}^2\vee G_{2p+2}^3$. Theorem \ref{n-1} implies that $\Delta(H)\leq n-2$. Denote by $t$ the number of vertices of degree $n-2$ in $H$. Using the Handshaking lemma, we have that $t(n-2) + (n-t)(n-3) \geq 2e = n(n-1)-n-2-2p$, which implies that  $t\geq n-2p-2$. 

Assume that $H$ has an induced subgraph $G_1'$ that is isomorphic to $G_1$ and that is induced by $n-2p-2$ vertices of degree $n-2$. Relabel the vertices of $H$ such that $V(H) =  V(G_1')\cup V(H\setminus G_1')$ in this order.
Therefore,
\begin{align*}
      Y^TA(H)Y - Y^TA(G)Y &= 2\left(\sum_{uv\in E(H)}y_uy_v - \sum_{uv\in E(G)}y_uy_v\right)\\
      &= 2\left(\sum_{uv\in E(H\setminus G_1')}y_uy_v - \sum_{uv\in E(G_2)}y_uy_v\right)\\
      &= 0,
    \end{align*}
where the last equality follows because the eigenvector $Y$ is constant (each entry is $x_2$) on $V(G_2)$ and the number of edges in the subgraph $H\setminus G_1'$ equals the number of edges in the subgraph $G_2$ of $G$. Hence, we have that 
 \begin{equation*}
  \rho(H)\geq Y^TA(H)Y=Y^TA(G)Y = \rho(G),
  \end{equation*}
 with equality if and only if $Y$ is also the principal eigenvector of $H$.
    
Assume that equality happens and $\rho(H)=\rho(G)$. Because $H$ is not of the form $G_{n-2(p+1)}^2\vee G_{2(p+1)}^3$, there exists a vertex $v\in V(H\setminus G_1')$ with $d(v)\not= n-3$. The eigenvalue-eigenvector equation $A(H)Y=\rho(H)Y$ for $v$  gives that 
\begin{equation*}
 (n-2p-2)x_1 + (d(v) - n+2p+2)x_2 = \rho(H)x_2.
 \end{equation*}
Similarly, the eigenvalue-eigenvector equation $A(G)Y=\rho(G)Y$ for a vertex in $V_2$ yields that
\begin{equation*}
(n-2p-2)x_1 + (2p-1)x_2 =\rho(G)x_2.
\end{equation*}
Because $\rho(H)=\rho(G)$, the above equations imply that $d(v)=n-3$, contradiction. Therefore, our assumption $\rho(H)=\rho(G)$ was incorrect and $Y$ is not the principal eigenvector of $H$. Hence, $\rho(H)>\rho(G)$. This finishes the proof of Claim 1.

 {\em Proof of Claim 2}  We use the notation from the previous proof and $x=(x_1,x_2)^T$ is the principal eigenvector of $Q_G$ corresponding to $\rho(G)$ (see \eqref{eq:qgthm310}). After solving the following eigenvalue-eigenvector equations     
 \begin{align*}
        &(n-2p-4)x_1 + (2p+2)x_2 = \rho(G)x_1\\
        &(n-2p-2)x_1 + (2p-1)x_2 =\rho(G)x_2,
\end{align*} 
we obtain that $\frac{x_1}{x_2}=\frac{2p+2}{\rho(G)-n+2p+4}$. Because $\rho(G)<n-2$ (see  \eqref{eq:rhoqg310}), we deduce that $\frac{2p+2}{\rho(G)-n+2p+4}>1$ and therefore, $x_1>x_2$.

Let $H$ be a minimizer in $\mathcal{G}_{n,e}$. Consider a subset $V_1(H)$ of $n-2p-2$ vertices of degree $n-2$ in $H$. Relabel the vertices of $H$ such that $V = V_1(H) \cup V\setminus V_1(H)$ in this order. By Claim 1, $H$ does not have $G_1$ as a subgraph induced by some of its vertices of degree $n-2$. Therefore any subset of $n-2p-2$ vertices of $H$ of degree $n-2$ induces a subgraph that has more edges than the number of edges in $G_1$. Assume that $|E(H[V_1])| = k  + |E(G_1)|$ for some $k\in \mathbf{N}$. Because the vertices of $G_1$ have degree $n-2$ in $G$, to add an edge in $G_1$, we need to delete two cross edges between $G_1$ and $G_2$ to maintain their degrees and add an edge to $G_2$ to keep the number of edges fixed. Therefore, 
\begin{align*}
  Y^TA(H)Y - Y^TA(G)Y &= 2\left(\sum_{uv\in E(H)}y_uy_v - \sum_{uv\in E(G)}y_uy_v\right)\\
 &= 2\left(\sum_{uv\in E(H[V_1(H)])}y_uy_v - \sum_{uv\in E(G_1)}y_uy_v\right)\\
 &+ 2\left( \sum_{uv\in E(H[V_1(H)], H[V\setminus V_1(H)])}y_uy_v - \sum_{uv\in E(G_1, G_2)}y_uy_v\right)\\
 &+ 2\left(\sum_{uv\in E(H[V\setminus V_1(H)])}y_uy_v - \sum_{uv\in E(G_2)}y_uy_v\right)\\
 &= 2kx_1^2  - 4kx_1x_2 + 2kx_2^2\\
 &= 2k(x_1-x_2)^2.
\end{align*}
Thus, we obtain that
\begin{equation*}
\rho(H)\geq Y^TA(H)Y = Y^TA(G)Y + 2k(x_1-x_2)^2 >\rho(G).
\end{equation*}
This completes the proof of Claim 2 and of Theorem 3.10.
\end{proof}

The proof of the following result (for odd $n$) is similar to the previous one and we will omit it.
\begin{theorem}\label{n-2 to n-3 2}
If $n\ge 5$ is odd, $1\leq p\leq \frac{n-3}{2}$ and $e = {n\choose 2} - (\frac{n+3}{2}+p)$, then $\rho_{min}(n,e) = \frac{n-5+\sqrt{n(n+2)-8p-3}}{2}$. The minimizers in $\mathcal{G}_{n,e}$ are of the form $G_{n-2(p+1)+1}^2\vee G_{2(p+1)-1}^3$.  
\end{theorem}


\subsection{The case $e={n-1\choose 2}-2$}

Let $G_1$ be any of the two graphs in Figure \ref{n6e8 minimizers}. Construct a new graph $G=(V,E)$ by taking the join of $G_1$ with a $(n-9)$-regular graph of order $n-6.$ The graph $G$ has order $n$ and size ${n-1\choose 2}-2.$
\begin{proposition}\label{n-3 regular -1}
For $n\geq 9$ and $e = {n-1\choose 2}-2$, $G = G_1 \vee G^3_{n-6}$ is the minimizer in $\mathcal{G}_{n,e}$ and $\rho_{min}(n,e)$ equals the largest root of the polynomial
$$P(x) = x^3 + x^2(7-n) + 4x(4-n) +6-2n.$$
\end{proposition}
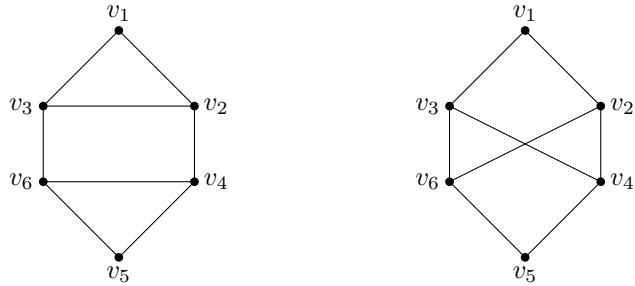
\begin{figure}[ht!]\label{fig:min_n6e8}
    \centering
\begin{tikzpicture}
\foreach \x in {(0,1), (0,-2), (1,-1),(1,0),(-1,0),(-1,-1)}{
 \draw[fill] \x circle[radius = 0.05cm]; }
 \draw (0,1)--(1,0)--(1,-1)--(0,-2)--(-1,-1)--(-1,0)--(0,1);
 \draw (-1,0)--(1,0);
 \draw(-1,-1)--(1,-1);
 \node[right] at (1,0) {$v_2$};
 \node[above] at (0,1) {$v_1$};
 \node[right] at (1,-1) {$v_4$};
 \node[left] at (-1,0) {$v_3$};
 \node[left] at (-1,-1) {$v_6$};
 \node[below] at (0,-2) {$v_5$};
\end{tikzpicture}
\hspace{2cm}
\begin{tikzpicture}
\foreach \x in {(0,1), (0,-2), (1,-1),(1,0),(-1,0),(-1,-1)}{
 \draw[fill] \x circle[radius = 0.05cm]; }
 \draw (0,1)--(1,0)--(1,-1)--(0,-2)--(-1,-1);
 \draw (-1,0)--(1,-1);
 \draw (-1,0)--(0,1);
 \draw(-1,-1)--(1,0);
\draw (-1,0)--(-1,-1);
 \node[right] at (1,0) {$v_2$};
 \node[above] at (0,1) {$v_1$};
 \node[right] at (1,-1) {$v_4$};
 \node[left] at (-1,0) {$v_3$};
 \node[left] at (-1,-1) {$v_6$};
 \node[below] at (0,-2) {$v_5$};
\end{tikzpicture}
 \caption{Minimizers on 6 vertices and 8 edges}
    \label{n6e8 minimizers}
\end{figure}
We will use the following lemma.
\begin{lemma}\label{exactly one n-2}
If $e<\binom{n}{2}-n$, then a minimizer graph in $\mathcal{G}_{n,e}$ cannot have exactly one vertex of degree $n-2$.
\end{lemma}
\begin{proof}
Assume that $e<\binom{n}{2}-n$. Let $G$ be a minimizer on $n$ vertices and $e$ edges. We use proof by contradiction. Suppose that $G$ has a vertex $v$ with degree $n-2$. Denote by $u$ the vertex of $G$ that is not adjacent to $v$. 

\noindent{\bf Case 1} Suppose $d(u)<n-3$. Then $G$ cannot be a minimizer by Lemma \ref{Kelmans1}.

\noindent{\bf Case 2} Suppose $d(u)=n-3$. Let $z$ be a non-neighbor of $u$ such that $z\neq v$. If $d(z)<n-3$, then $G$ cannot be a minimizer by Lemma \ref{Kelmans1}. If $d(z) =n-3$, then let $w_1\in N(v)$ such that $z\not\sim w_1.$

If $d(w_1)\leq n-4.$ Choose another vertex $y\in V\setminus\{z\}$ such that $w_1\not\sim y$  with least possible degree. Delete the edge $vy$ and add $yw_1$  to obtain a new graph $G'$. Let $x$ be the principal eigenvector of $G'$. If $x_{w_1}\leq x_v$, then $\rho(G)>\rho(G')$ by Theorem \ref{rotation} and we are done. Else if $x_{w_1}>x_v$, we show $(x_{w_1} - x_v)(\sum_{\beta\in N(v)-N(w_1)-w_1}x_{\beta} - x_u)\geq 0.$ Let $\rho =\rho(G')$ be the spectral radius of $G'$. We note that
\begin{align*}
      (x_{w_1} - x_v)\left(\sum_{{\scriptscriptstyle \beta\in N(v)-N(w_1)-w_1}}x_{\beta} - x_u\right)\geq 0
      &\iff\rho\left(\sum_{{\scriptscriptstyle \beta\in N(v)-N(w_1)-w_1}}x_{\beta} - x_u\right)\geq 0\\
      &\iff \rho\left(\sum_{{\scriptscriptstyle \beta\in N(v)-N(w_1)-w_1 -z}}x_{\beta} \right)+\rho(x_z -x_u) \geq 0\\
      &\iff \rho\left(\sum_{{\scriptscriptstyle\beta\in N(v)-N(w_1)-w_1 -z}}x_{\beta} \right) +(x_v - x_{w_1}) \geq 0.
  \end{align*}
Since after the edge rotation $d(w_1)\leq n-3, N(v)-N(w_1)-w_1 -z\neq\emptyset$. Let $t\in N(v)-N(w_1)-w_1 -z$. To prove the last inequality above, it suffices to show $(x_v-x_{w_1}) +\rho x_t\geq 0.$ Which is true as  $$(x_v -x_{w_1}) +\rho x_t \geq (x_v-x_{w_1}) + x_v +x_z +x_u\geq (x_z-x_{w_1}) +x_u\geq x_u(1-\frac{1}{\rho})> 0.$$

If $d(w_1)=n-3.$ Let $w_2\not=z\in V$ such that $w_2\not\sim w_1.$ If $d(w_2)\leq n-4.$  Choose another vertex $y\in V\setminus\{w_1\}$ such that $w_2\not\sim y$  with least possible degree  and do the $vy$ edge rotation to $w_2y$ to obtain a new graph $G'$. Let $x$ be the principal eigenvector of $G'$.  If $x_{w_2}\leq x_v$, we are done by Theorem \ref{rotation}. Else if $x_{w_2}>x_v$, we show $(x_{w_2} - x_v)(\sum_{\beta\in N(v)-N(w_2)-w_2}x_{\beta} - x_u)\geq 0.$ Let $\rho = \rho(G')$ be the spectral radius of $G'$. We note that
\begin{align*}
      (x_{w_2} - x_v)&\left(\sum_{\beta\in N(v)-N(w_2)-w_2}x_{\beta} - x_u\right)\geq 0\\
      \iff\rho &\left(\sum_{\beta\in N(v)-N(w_2)-w_2}x_{\beta} - x_u\right)\geq 0\\
      \iff \rho &\left(\sum_{\beta\in N(v)-N(w_2)-w_2 -w_1}x_{\beta} \right)+\rho(x_{w_1} -x_u) \geq 0\\
      \iff \rho &\left(\sum_{\beta\in N(v)-N(w_2)-w_2 -{w_1}}x_{\beta} \right) +(x_v +x_u  - x_{w_1} -x_{w_2}) \geq 0.
  \end{align*}
Since after edge rotation $d(w_2)\leq n-3, N(v)-N(w_2)-w_2 -w_1\neq\emptyset$. Let $t\in N(v)-N(w_2)-w_2 -w_1.$ To prove the last inequality above, it suffices to show $(x_v+x_u-x_{w_1}-x_{w_2}) +\rho x_t\geq 0.$ Which is $$(x_v+x_u -x_{w_1}-x_{w_2}) +\rho x_t \geq (x_v+x_u-x_{w_1}-x_{w_2}) + x_v +x_z +x_u + x_{w_1}= 2(x_v+x_u) +x_z - x_{w_2}> 0$$ as $N(w_2)\subset N(v)\cup N(u)\cup N(z).$

If $d(w_2)=n-3$, we keep repeating the same process as above until at some $k^{th}$ iteration we have $d(w_k)\leq n-4.$ Then either $x_{w_k}\leq x_v$ in the new graph $G'$ obtained by edge rotation, or else we argue $(x_{w_k}-x_v)(\sum_{\beta\in N(v)-N(w_k)-w_k}x_\beta-x_u)\geq 0.$ Which is true as $\rho(\sum_{\beta\in N(v)-N(w_k)-w_k}x_\beta-x_u)\geq (x_v+x_u+x_z -x_{w_{k-2}}-x_{w_{k-1}}-x_{w_k})+x_v+x_u+x_z+\sum_{i=1}^{k-1}x_{w_i}= 2(x_v+x_u+x_z)+\sum_{i=1}^{k-3}x_{w_i} -x_{w_k}\geq 0.$ The process stops at most at the $(n-3)^{th}$ iteration with $d(w_{n-3})\leq n-4.$ Because if $d(w_{n-3}) = n-3$ too, then $G$ has $(n-1) $ vertices of degree $n-3$ and 1 vertex of degree $n-2$, which is not possible as $e<{n\choose 2} -n.$
\end{proof}

\begin{proposition}
 For $n\geq 9$ and $e = {n-1\choose 2}-2$, if $G$ is a minimizer in $\mathcal{G}_{n,e}$, then $\Delta(G)-\delta(G)= 1.$ 
\end{proposition}
\begin{proof} Let $G=(V,E)$ be $G_1\vee G^3_{n-6}.$ Consider the partition of $G$ into three parts : $\{v_1, v_5\}, \{v_2, v_3, v_4, v_6\},$ and $ V(G^3_{n-6}).$ The quotient matrix of this partition is
\begin{equation*}
Q_G = \begin{bmatrix}
        0&2&n-6\\
        1&2&n-6\\
        2&4&n-9
    \end{bmatrix}
\end{equation*}    
    and its characteristic polynomial is
    \begin{equation}
      P(x) = x^3 + x^2(7-n) + 4x(4-n) +6-2n.  
    \end{equation}
    Since the partition is equitable, Theorem \ref{thm:eq_sr} implies that $\rho(G)=\rho(Q_G)$.  Let $X = (x_1,x_2,x_3)^T$ be the principal eigenvector of $Q_G.$ After solving the following eigenvalue-eigenvector equations,     
  \begin{align*}
        \rho(G)x_1&=2x_2+(n-6)x_3,\\
        \rho(G)x_2&=x_1+2x_2+(n-6)x_3,\\
        \rho(G)x_3&=2x_1+4x_2+(n-9)x_3,
    \end{align*}
we find that
     \begin{equation}\label{eigenvector for n-3 minus an edge}
        x_2>x_3>x_1.
    \end{equation}

Suppose $G' = (V' E')$ is a minimizer with $u,v\in V'$ of degree $n-2$. Consider the partition of $G'$ in to two parts : $\{u,v\}$ and $V'\setminus\{u,v\}$. 
If $u\not \sim v$, the  quotient matrix is
\begin{equation*}
Q_1 = \begin{bmatrix}
    0&n-2\\
    2&\frac{n(n-7)+6}{n-2}
\end{bmatrix}
\end{equation*}
and if $u\sim v$, the quotient matrix is
\begin{equation*}
Q_2 = \begin{bmatrix}
   1&n-3\\
   \frac{2(n-3)}{n-2}&\frac{n(n-7)+8}{n-2}
\end{bmatrix}.
\end{equation*} 
Their respective largest eigenvalues are
$$\rho(Q_1) = \frac{n^2-7n+6 +\sqrt{n^4-6n^3+13n^2+12n-28}}{2(n-2)}$$
and
$$\rho(Q_2) = \frac{n^2-6n+6 +\sqrt{n^4-8n^3+20n^2+8n-44}}{2(n-2)}.$$ We note that $\rho(Q_1)\leq \rho(Q_2)$ and plugging $\rho(Q_1)$ in $P(x)$ we check $P(\rho(Q_1))$ is monotonically decreasing in $n$ with the limit 0. Since two of the three roots of $Q_G$ are negative, we get $P(x)>0$ for all $x\geq \rho(Q_1)$. Hence we conclude $\rho(Q_1)>\rho(G),  \rho(Q_2)>\rho(G)$, and by eigenvalue interlacing (see Theorem \ref{thm:int_sr}), $\rho(G')>\rho(G).$ In view of this, together with Lemma \ref{exactly one n-2}, we get if $G'$ is a minimizer then $\Delta(G')\leq n-3.$
By handshaking lemma, we also get $\delta(G')\geq n-5.$ In particular, the only possibility is when $G'$ has $n-1$ vertices of degree $n-3$ and one vertex, say $v$ of degree $n-5$. We now show that such a graph cannot be a minimizer in $\mathcal{G}_{n,e}$. Consider the partition of $G'$ in to two parts: $\{v\}$ and $V'\setminus\{v\}.$ The corresponding quotient matrix and its
characteristic polynomial respectively are
$$Q = \begin{bmatrix}
    0&n-5\\
    \frac{n-5}{n-1}&\frac{n(n-5)+8}{n-1}
\end{bmatrix}$$ and
$$g(x) = x\left(x - \frac{n(n-5)+8}{n-1}\right) - \frac{(n-5)^2}{n-1}.$$ The largest eigenvalue of $Q$ is
$$\rho(Q) = \frac{n^2-5n+8 +\sqrt{n^4-6n^3-3n^2+60n-36}}{2(n-1)}.$$ We note that $\rho(Q)\geq \rho(Q_1)$. Therefore, $\rho(Q)>\rho(G)$ and by eigenvalue interlacing (see Theorem \ref{thm:int_sr}), $\rho(G')>\rho(G).$ This completes the proof.
\end{proof}
\begin{proposition}
      If $n\geq 9$ and $e = {n-1\choose 2}-2$, then a minimizer in $\mathcal{G}_{n,e}$ is a $(n-3)$-regular graph minus an edge.
\end{proposition}
\begin{proof}
    Let $G' = (V',E')$ be a minimizer with $x,y\in V'$ of degree $n-4$ such that $x\sim y$ in $G'$ Consider the partition of $G'$ into two parts: $\{x,y\}$ and $V'\setminus\{x,y\}.$
    The corresponding quotient matrix and its characteristic polynomial respectively are 
    $$Q = \begin{bmatrix}
        1&n-5\\2\frac{n-5}{n-2}&\frac{n(n-7)+16}{n-2}
    \end{bmatrix}$$
    and
    $$g(x) = (x-1)\left(x-\frac{n(n-7)+16}{n-2}\right) - \frac{2(n-5)^2}{n-2}$$
    with $$\rho(Q) = \frac{n^2-6n+14+\sqrt{n^4-8n^3+4n^2+72n-76}}{2(n-2)}.$$ Consider the polynomial $f(x) = P(x)-xg(x).$ We note that $$f'(x) = \frac{2n(3x-n)+11n+2}{n-2}>0$$ in the interval $[\rho(Q), n-3]$ and so $f(x)$ is increasing in the interval. When $x = \rho(Q), f(\rho(Q)) = P(\rho(Q)),$ where $P(\rho(Q))$ is monotonically decreasing in $n$ with the limit 0. Therefore $f(x)>0$ in the interval $[\rho(Q), n-3].$ This proves $\rho(G)\leq \rho(Q)$ and by eigenvalue interlacing (see Theorem \ref{thm:int_sr}), $\rho(G)<\rho(G').$ This completes the proof.
\end{proof}

\begin{proof}[Proof of Proposition \ref{n-3 regular -1}]
Let $G = G_1\vee G^3_{n=6}$ be the graph as defined in the statement and let $G' = (V',E')$ be a minimizer in $\mathcal{G}_{n,e}$. We will show that $G'\cong G.$ Relabel the vertices of $G'$ such that the two vertices of degree $n-4$ coincides with $v_1, v_5$ of $G_1.$ Since $v_1\not\sim v_5$ in $G'$, we note that $|N(v_1)\cap N(v_5)|\geq n-6$. Relabel again the vertices of $G'$ other than $v_1$ and $v_5$ such that the first $n-6$ vertices of the set $N(v_1)\cap N(v_5)$ coincides with $V(G^3_{n-6})$ of $G.$ If in $G'$ vertices $V(G^3_{n-6})$ are not in join with $\{v_2, v_3, v_4, v_6\}$, then by Equation \ref{eigenvector for n-3 minus an edge} and Theorem \ref{local switching}, we get $\rho(G')>\rho(G),$ a contradiction. The proof that the subgraph induced by $\{v_1, v_2, v_3, v_4, v_5, v_6\}$ in $G'$ is isomorphic to $G_1$ is similar to the proof of Proposition \ref{min for e=n^2/4-1}, hence we omit that calculation. This completes the proof.
\end{proof}

\section{Sporadic cases}

\subsection{The case $e=\left(\frac{n-1}{2}\right)\left(\frac{n+1}{2}\right)$}
\begin{proposition}\label{similar degree}
 For fixed $n,e$, and $p\geq 2$, consider the set $S = \{(d_1, d_2,\cdots , d_n)\; :\; \sum_i d_i = 2e, d_i\in\mathbb{N}\}$. 
 The $p$-mean is minimum for a degree sequence in which all $d_i's$  are as equal as possible, i.e., $d_i = \floor{2e/n}$ or $d_i = \ceil{2e/n}.$   
\end{proposition}
\begin{proof}
It suffices to show that for all $i,j\in [n]$, $|d_i-d_j|<2$ in the degree sequence $d$ for which $d^{(p)}$ is minimum. Suppose for contradiction that $d$ is such a degree sequence with $d_i\geq d_j +2$ for some $i,j\in [n].$ Generate a new degree sequence $d'$ from $d$
such that all entries are same except $d'_i = d_i -1$ and $d'_j = d_j +1.$ We show that $(d_i-1)^p + (d_j+1)^p < d_i^p + d_j^p$ to get a contradiction.

Consider the following differentiable function $f(h) = (d_i - h)^p + (d_j +h)^p.$  We note that $f(h)$ is a strictly decreasing function for $h\in (0,\frac{d_i - d_j}{2})$ and so  $(d_i-1)^p + (d_j+1)^p < d_i^p + d_j^p$. This completes the proof.
\end{proof}
    This means for graphs on $n$ vertices and $e$ edges, if the minimizer $G_0$ has $char_{\rho}(G_0) = p$, then any graph $G\in\mathcal{G}_{n,e}$ with $Ir(G)\leq 1$ has $char_{\rho}(G)\geq p.$ 
\noindent The following result was proved in \cite{Stanic}.
\begin{proposition}\label{e=s(s+1)}
 For $(n,e) = (2k+1, k(k+1))$, $G =K_{k,k+1}$ is the minimizer of $\mathcal{G}_{n,e}$.
\end{proposition}
\begin{proof}
 We know that the spectral radius of $K_{k,k+1}$ is $ \sqrt{k(k+1)}$ and the 2-mean of its degree sequence
 \begin{align*}
  d^{(2)} &= \sqrt{\frac{1}{2k+1}(k(k+1)^2 +(k+1)k^2)}\\ &= \sqrt{\frac{(2k+1)k(k+1)}{2k+1}}\\
  &= \sqrt{k(k+1).}
 \end{align*}
 Therefore, its spectral mean characteristic equals 2 and by Theorem \ref{charG}, we get  $K_{k,k+1}$ is a minimizer. The uniqueness follows from Theorem \ref{semiregular}.  
\end{proof}

\subsection{The case $e = \frac{n^2}{4} -1$}

In this subsection, we will prove that for $n$ even and $e = \frac{n^2}{4}-1,$ $K_{\frac{n}{2},\frac{n}{2}}$ minus an edge is a minimizer in $\mathcal{G}_{n,e}$ and $$\rho_{min}(n,e) = \frac{n-2+\sqrt{n^2 +4n-12}}{4}.$$ 

\begin{remark}
This minimizer is not unique. For example, when $n=6$, there are 2 non-isomorphic minimizers, see Figure \ref{n6e8 minimizers}, and when $n=8,$ there are 5 non-isomorphic minimizers. See Figure \ref{n8e15minimizers}. We characterize all the minimizers in Proposition \ref{min for e=n^2/4-1}.
\end{remark}
\begin{figure}[ht!]
 \centering
    \includegraphics[scale =0.3]{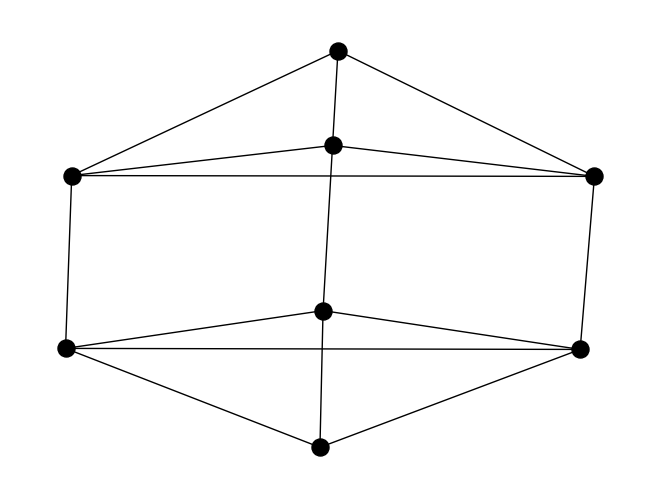}
    \includegraphics[scale =0.3]{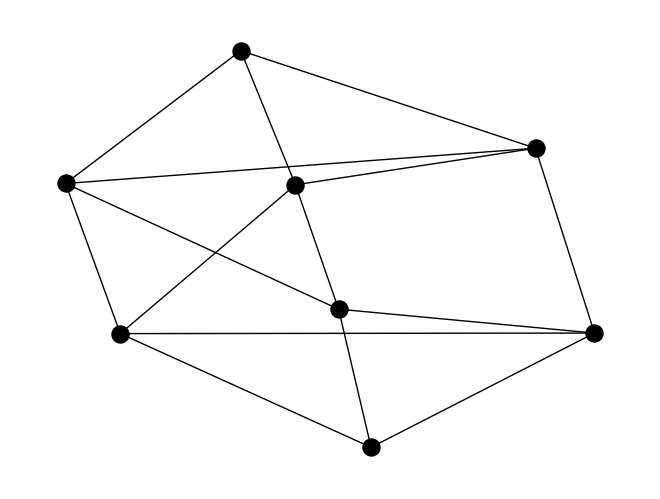}
    \includegraphics[scale =0.3]{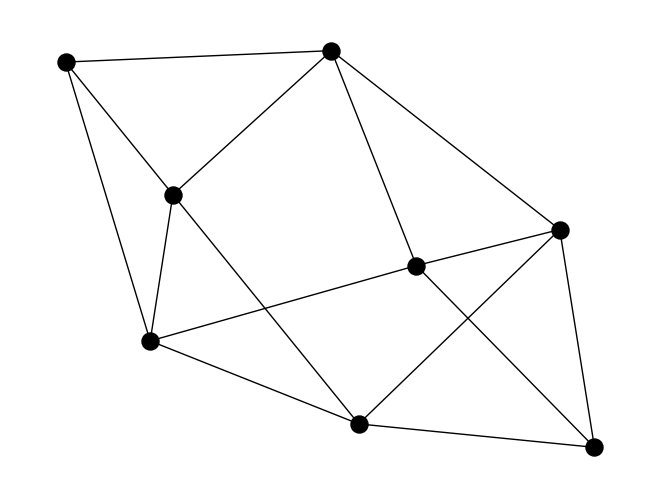}
    \includegraphics[scale =0.3]{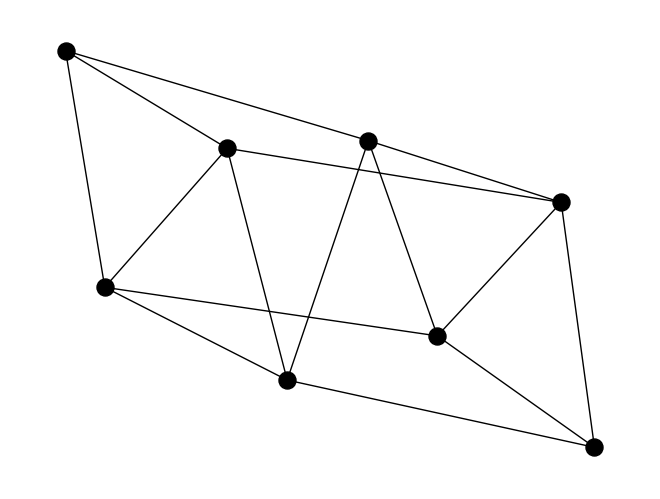}
    \includegraphics[scale =0.3]{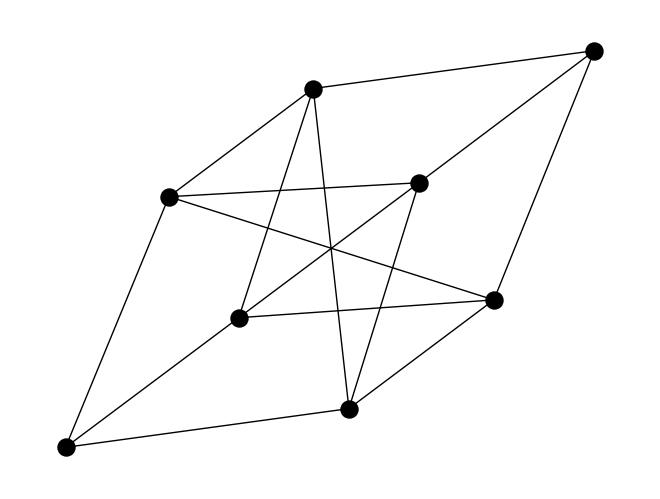}
    \caption{All $n=8,e=15$ non-isomorphic minimizers.}
    \label{n8e15minimizers}
\end{figure}

\begin{proposition}
    For $n$ even and $e = \frac{n^2}{4}-1,$ a minimizer $G$ satisfies $\Delta(G)-\delta(G) =1.$ 
\end{proposition} 
\begin{proof}
  Suppose $G$ is a minimizer in $\mathcal{G}_{n,e}$ with a vertex $i$ of degree $d_i=\frac{n}{2} + k,$ $k\in\mathbb{N}$. By Proposition \ref{similar degree}, the sum $\sum_{j\not =i, j=1}^n d_j^2$ is minimum when all $d_j \in \{\lfloor \frac{n}{2} -\frac{2+2k}{n-1}\rfloor = \frac{n}{2}-1, \lceil  \frac{n}{2} -\frac{2+2k}{n-1} \rceil = \frac{n}{2}\}$. Therefore, by Lemma \ref{increasing p}, we have
  $$\rho(G)\geq \sqrt{\frac{1}{n}\left( (n-3-k)\frac{n^2}{4} + (2+k)(\frac{n}{2}-1)^2 + (\frac{n}{2}+k)^2 \right)}.$$ We now show for $k\in \mathbb{N},$
  \begin{align*}
    &\sqrt{\frac{1}{n}\left( (n-3-k)\frac{n^2}{4} + (2+k)(\frac{n}{2}-1)^2 + (\frac{n}{2}+k)^2 \right)}>   \frac{n-2+\sqrt{n^2 +4n-12}}{4}\\
    &\iff (n-3-k)\frac{n^2}{4} + (2+k)(\frac{n}{2}-1)^2 + (\frac{n}{2}+k)^2> n\left( \frac{n-2+\sqrt{n^2 +4n-12}}{4}\right)^2\\
    &\iff \frac{n^3}{4} -2n +k^2 +k +2> \frac{n^3}{8} -\frac{n}{2} + \frac{n(n-2)\sqrt{n^2+4n-12}}{8}.
  \end{align*}
 Observing the function $f(k)  = \frac{n^3}{4} -2n +k^2 +k +2$ is increasing in $k$, it suffices to show $f(1) > \frac{n^3}{8} -\frac{n}{2} + \frac{n(n-2)\sqrt{n^2+4n-12}}{8}$ to prove the above inequality. We note that
 \begin{align*}
     &\hspace{1.3cm} f(1) = \frac{n^3}{4} -2n + 4> \frac{n^3}{8} -\frac{n}{2} + \frac{n(n-2)\sqrt{n^2+4n-12}}{8}\\
     &\iff \left(\frac{n^3}{8} -\frac{3n}{2} + 4\right)^2 - \left( \frac{n(n-2)\sqrt{n^2+4n-12}}{8}\right)^2>0\\
     &\iff 3n^2-12n+16>0
 \end{align*}
 which is true for all $n\geq 2.$ Therefore, $\Delta(G)\leq \frac{n}{2}.$ In view of this, the only remaining case to rule out is when one vertex has degree $\frac{n}{2}-2$ and remaining vertices have degree $\frac{n}{2}.$ Similar to above, we note that
 \begin{align*}
  &\hspace{1.2cm}\sqrt{\frac{1}{n}\left( (\frac{n}{2}-2)^2 +(n-1)(\frac{n}{2})^2\right)}> \frac{n-2+\sqrt{n^2 +4n-12}}{4}&\\
  &\iff (\frac{n}{2}-2)^2 +(n-1)(\frac{n}{2})^2>  n\left( \frac{n-2+\sqrt{n^2 +4n-12}}{4}\right)^2&\\
   &\iff\frac{n^3}{4} -2n + 4> \frac{n^3}{8} -\frac{n}{2} + \frac{n(n-2)\sqrt{n^2+4n-12}}{8}&\\
     &\iff \left(\frac{n^3}{8} -\frac{3n}{2} + 4\right)^2 - \left( \frac{n(n-2)\sqrt{n^2+4n-12}}{8}\right)^2>0 &\\
     &\iff 3n^2-12n+16>0
 \end{align*}
  which is true for all $n\geq 2.$ Therefore, $\delta(G)\geq \frac{n}{2}-1.$ Since $n\not | \; 2e$, $\Delta(G)-\delta(G) =1$.
\end{proof}

\begin{proposition}
For even $n$ and $e = \frac{n^2}{4}-1,$ if $G$ is minimizer in $\mathcal{G}_{n,e}$, then $G$ is obtained from an $\frac{n}{2}$-regular graph on $n$ vertices by deleting an edge. 
\end{proposition}
\begin{proof}
Let $G$ be a minimizer and let $t$ be the number of vertices in $G$ of degree $\frac{n}{2}-1$. We have $ \frac{n^2}{2}-2 = 2e =\sum_{v\in V}d(v) = t(\frac{n}{2}-1) + (n-t)\frac{n}{2} = -t + \frac{n^2}{2}.$ Therefore, $t=2$ and let $u,v\in V$ be the two vertices of degree $\frac{n}{2}-1.$ Suppose $u\sim v$ in $G.$ Consider the partition of $G$ into the following two sets: $\{u,v\}, V\setminus \{u,v\}.$ The quotient matrix of this partition is 
$$Q_G =\begin{bmatrix}
    1&\frac{n}{2}-2\\
    \frac{n-4}{n-2}&\frac{n^2-4n+8}{2(n-2)}
\end{bmatrix},$$ and the corresponding characteristic polynomial is
\begin{equation}
 Q(x) = (x-1)(x-\frac{n^2-4n+8}{2(n-2)}-\frac{n-4}{n-2}(\frac{n}{2}-2). 
\end{equation}
Let $H = K_{\frac{n}{2},\frac{n}{2}} - uv$, where $uv$ is an edge in $K_{\frac{n}{2},\frac{n}{2}}$. Consider the same partition of vertex set of $H$ as above of $G$. The quotient matrix of this partition is 
$$Q_H =\begin{bmatrix}
    0&\frac{n}{2}-1\\
    1&\frac{n^2-4n+4}{2(n-2)}
\end{bmatrix},$$
and the corresponding characteristic polynomial is
\begin{equation}
 P(x) = x^2 -x\left(\frac{n^2-4n+4}{2(n-2)}\right)-\frac{n}{2}+1.   
\end{equation}
\begin{align*}
P(x)-Q(x)&= x\left(1 + \frac{n^2-4n+8}{2(n-2)} - \frac{n^2-4n+4}{2(n-2)}\right)-\frac{n^2-4n+8}{2(n-2)}-3\frac{n}{n-2} +1 +\frac{8}{n-2}\\ 
 &= \frac{2xn -n^2 +4}{2(n-2)}
\end{align*}
When $x = \rho(H) = \frac{n-2 +\sqrt{n^2+4n-12}}{4}$,
    \begin{align*}
      (P(x)-Q(x))& = \frac{(n-2 +\sqrt{n^2+4n-12})n-2n^2+8}{4(n-2)}\\
      -Q(x)&=\frac{-n^2-2n+8+n\sqrt{n^2+4n-12}}{4(n-2)}> 0.
    \end{align*}
    The last inequality follows from $n^2(n^2+4n-12)- (n^2+2n-8)^2 = n-2> 0$ for $n> 2.$ Therefore, the largest root of $Q(x)$ is greater than $\rho(H)$, and by eigenvalue interlacing (see Theorem \ref{thm:int_sr}), $\rho(G)>\rho(H).$ This completes the proof.
\end{proof}
\noindent For $n$ even take a  $(\frac{n}{2}-1)$-regular graph $H$ on $n-2$ vertices. Divide the vertex set of $H$ into two equal parts and call the induced subgraphs $G_1$ and $G_2$. Construct a new graph $G$ by connecting a new vertex $u$ with all the vertices of $G_1$ and connecting another new vertex $v$ with all the vertices of $G_2$. The graph $G$ has order $n$ and size $\frac{n^2}{4}-1.$ 
\begin{proposition}\label{min for e=n^2/4-1}
 For even $n$ and $e = \frac{n^2}{4}-1$,  the graph $G$ as described above is a minimizer in $\mathcal{G}_{n,e}$ and $$\rho_{min}(n,e) = \frac{n-2+\sqrt{n^2 +4n-12}}{4}.$$ 
\end{proposition}
\begin{proof}
Consider the partition of $G$ into two parts : $\{u,v\}, V(H).$ The corresponding quotient matrix is
\begin{equation*}
Q_G =\begin{bmatrix}
        0&\frac{n}{2}-1\\
        1&\frac{n}{2}-1
    \end{bmatrix}.
\end{equation*}
The largest eigenvalue $\rho(Q_G)$ of $Q_G$ is $\frac{n-2+\sqrt{n^2 +4n-12}}{4}$. Because this partition is equitable, Theorem \ref{thm:eq_sr} implies that $\rho(G)=\rho(Q_G).$ Let $x=(x_1, x_2)^T$ be the principal eigenvector of $Q_G.$ Then $x = (x_1, x_1, x_2, x_2, \ldots, x_2)^T$ is an eigenvector of $A(G)$ corresponding to $\rho(G)$. Let $X = \frac{X}{||X||_2}$. Assume that $G'=(V',E')$ is a minimizer in $\mathcal{G}_{n,e}$. Let $u,v\in V'$ with $d(u)= d(v) =\frac{n}{2}-1.$ It suffices to show $N(u)\cap N(v)=\emptyset$. 

Assume that $|N(u)\cap N(v)| = k>0$. Relabel the vertices of $G'$ such that $V' = \{u,v, N(u), N(v)\setminus N(u), \tilde{V}\}$, where $\tilde{V} = V'\setminus (\{u,v\}\cup N(u)\cup N(v))$, in this order. Suppose $N(u)\cap N(v) = \{y_1, y_2, ..., y_k\}$ and $\tilde{V} = \{z_1, z_2, ..., z_k\}$. We note that $$\rho(G')\geq X^TA(G')X = X^TA(G)X + 2x_v\sum_{i=1}^k(x_{y_i}-x_{z_i})= \rho(G)+2x_1(0) = \rho(G).$$ Since $G'$ and $G$ are not isomorphic, the inequality above is strict. This completes the proof.
\end{proof}

\begin{remark}
We note that $K_{\frac{n}{2},\frac{n}{2}}+$ edge is not necessarily a minimizer for $e=\frac{n^2}{4}+1$. For instance, when $n=8$ and $e = 17$, $K_{4,4} +$ edge is not a minimizer as its index is 4.293, whereas the $\rho_{min}(8,17) \approx 4.281$ is attained by the graph in Figure \ref{n8e17}.
\end{remark}

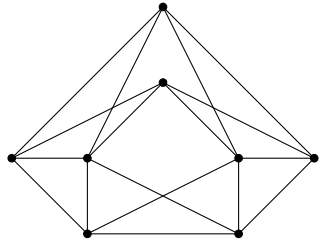
\begin{figure}[ht!]
    \centering
\begin{tikzpicture}
\foreach \x in {(0,0),(2,0),(0,1),(1,2),(2,1),(3,1),(-1,1),(1,3)}{
 \draw[fill] \x circle[radius = 0.05cm]; }
 \draw (1,2)--(0,1)--(0,0)--(2,0)--(2,1)--(3,1)--(1,2);
 \draw (0,1)--(-1,1)--(1,2)--(2,1)--(0,0);
 \draw(-1,1)--(1,3)--(3,1)--(2,0)--(0,1);
 \draw(0,0)--(-1,1);
 \draw (0,1)--(1,3)--(2,1);
\end{tikzpicture}
    \caption{The minimizer for $n=8, e=17.$}
    \label{n8e17}
\end{figure}

\subsection{The case $e = \frac{n^2}{3} -1$}
For $n$ divisible by 3, take a  $(\frac{n}{3}-1)$-regular graph $H$ of order $\frac{2n}{3}-2$. Divide the vertex set of $H$ into two equal parts and call the induced subgraphs $G_2, G_3$. Construct a new graph $\hat{G}$ by connecting a new vertex $x$ with all the vertices of $G_2$ and connecting another new vertex $y$ with all the vertices of $G_3$. Consider the graph $G = G_1\vee \hat{G}$, where $G_1 = \frac{n}{3}K_1$. The graph $G$ has order $n$ and size $\frac{n^2}{3}-1.$ 
\begin{proposition}\label{minimizer for e=n^2/3-1}
    For $n$ divisible by 3, $e=\frac{n^2}{3}-1$, graph $G$ as described above is the minimizer with
    $\rho_{min}(n,e)$ equals the largest root of the polynomial $$P(x) = x^3 +x^2\left(1-\frac{n}{3}\right) + x\left(1-\frac{2n^2}{9}- \frac{n}{3}\right) -\frac{2n^2}{9} +\frac{2n}{3}.$$ 
\end{proposition}
We will first prove the following.
\begin{proposition}\label{e=n^2/3-1}
    For $n$ divisible by 3, $e=\frac{n^2}{3}-1$, a minimizer is a $(\frac{2n}{3})$-regular graph minus an edge.  
\end{proposition}
\begin{proof}
    Let $G =(V,E)$ be a graph as described above. Consider the partition of $V$ into three parts : $V(G_1), V(G_2)\cup V(G_3)$, and $\{x,y\}.$ The corresponding quotient matrix is $$Q_G = \begin{bmatrix}
        0&\frac{2n}{3}-2&2\\
        \frac{n}{3}&\frac{n}{3}-1&1\\
        \frac{n}{3}&\frac{n}{3}-1&0
    \end{bmatrix}.$$
    Since the partition is equitable, Theorem \ref{thm:eq_sr} implies that $\rho(G)=\rho(Q_G)$, which is the largest root of the polynomial 
    \begin{equation}\label{char poly for e=n^2/3-1}
     P(x) = x^3 +x^2\left(1-\frac{n}{3}\right) + x\left(1-\frac{2n^2}{9}- \frac{n}{3}\right) -\frac{2n^2}{9} +\frac{2n}{3}.   
    \end{equation}
Let $v = (x_1,x_2,x_3)^T$ be the principal eigenvector of $Q_G.$ After solving the following eigenvalue-eigenvector equations
\begin{align*}
       2 \left(\frac{n}{3}-1\right)x_2 + 2x_3 &=\rho(G)x_1\\
        \frac{n}{3}x_1 +\left(\frac{n}{3}-1\right)x_2+x_3&=\rho(G)x_2\\
        \frac{n}{3}x_1 + \left(\frac{n}{3}-1\right)x_2 &=\rho(G)x_3,
    \end{align*}
 we get that
    \begin{equation}\label{eigenvector entries}
       x_2>x_1>x_3. 
    \end{equation}    
    
    Let $G' =(V', E')$ be a minimizer in $\mathcal{G}_{n,e}$. We first show that $\Delta(G')\leq \frac{2n}{3}+1$ and $\delta(G')\geq \frac{2n}{3}-1.$ Suppose there is a $ v\in V'$ with $d(v) = \frac{2n}{3}+k$, $k\geq 2.$ Consider the partition of $V'$ into two parts : $\{v\}, V'\setminus\{v\}.$ The corresponding quotient matrix is
    $$Q_{G'} = \begin{bmatrix}
        0&\frac{2n}{3}+k\\
        \frac{2n+3k}{3(n-1)}&\frac{2n^2-6-6k-4n}{3(n-1)}
    \end{bmatrix}$$ and its characteristic polynomial is 
    \begin{equation}
        g(x) = x\left(x-\left(\frac{2n^2-6-6k-4n}{3(n-1)}\right)\right)-\frac{(\frac{2n}{3}+k)^2}{n-1}.
    \end{equation}
    Consider the function \begin{align*}
        f(x) &= P(x)-xg(x)\\
        &= x^2\left(1-\frac{n}{3}+\frac{2n^2-6-6k-4n}{3(n-1)}\right)+x\left(1-\frac{2n^2}{9}-\frac{n}{3}+\frac{(\frac{2n}{3}+k)^2}{n-1}\right)-\frac{2n^2}{9} + \frac{2n}{3}.
    \end{align*}
Since $g\left(\frac{2n}{3}-\frac{2}{n}\right) = \frac{(kn+2)^2}{n^2(1-n)}<0,$ showing $f(x)>0$ in the interval $[\frac{2n}{3}-\frac{2}{n}, \frac{2n}{3}]$ will prove the claim. We note that $f(x)$ increases with $k$ as the terms involving $k$, $k(3k +2(2n-3x))>0$ for $x\in [\frac{2n}{3}-\frac{2}{n}, \frac{2n}{3}]$. So it is sufficient to work with $k=2.$ When $k=2,$
    $$f(x) = x^2\left(1-\frac{n}{3}+\frac{2n^2-18-4n}{3(n-1)}\right)+x\left(1-\frac{2n^2}{9}-\frac{n}{3}+\frac{(\frac{2n}{3}+2)^2}{n-1}\right)-\frac{2n^2}{9} + \frac{2n}{3},$$
    and $$f'(x)  = \frac{2n^2(3x-n)+3n^2 -126x+36n+27}{9(n-1)}.$$ For $n\geq 10$ and $x\in [\frac{2n}{3}-\frac{2}{n}, \frac{2n}{3}]$, $f'(x)>0$. When $x = \frac{2n}{3}-\frac{2}{n}$, $$f(x) = f\left(\frac{2n}{3}-\frac{2}{n}\right) = \frac{2}{3(n-1)}\left(n +18 -\frac{9}{n}-\frac{42}{n^2}\right)>0,$$ and when $x = \frac{2n}{3}$, $$f(x) = f\left(\frac{2n}{3}\right) = \frac{4n(n+3)}{9(n-1)}>0.$$ This proves that $\rho(G)<\rho(Q_{G'})$ and by eigenvalue interlacing (see Theorem \ref{thm:int_sr}), $\rho(G)<\rho(G').$ Therefore, if $G'$ is a minimizer, then $\Delta(G')\leq \frac{2n}{3}+1.$ The proof for $\delta(G')\geq \frac{2n}{3}-1$ has the similar calculation.\\

    Next we will show that if $G' = (V', E')$ is a minimizer, then it cannot have two or more vertices of degree $\frac{2n}{3}+1.$ Suppose $u,v\in V'$ with $d(u)=d(v) = \frac{2n}{3}+1.$ Consider the partition of $G'$ into two parts : $\{u,v\}, V\setminus\{u,v\}.$ The corresponding quotient matrices are  
    $$Q_1 = \begin{bmatrix}
        0&\frac{2n}{3}+1\\
        \frac{4n+6}{3(n-2)}&\frac{2n^2-8n-18}{3(n-2)}
    \end{bmatrix}, Q_2 =\begin{bmatrix}
        1&\frac{2n}{3}\\
        \frac{4n}{3(n-2)}&\frac{2n^2-8n-12}{3(n-2)}
    \end{bmatrix}$$
for when $u\not\sim v$ and $u\sim v$, respectively. Their respective largest eigenvalues are
$$\rho(Q_1) =\frac{n^2-4n-9+\sqrt{n^4+6n^2+42n+45}}{3(n-2)},$$
and
$$\rho(Q_2) = \frac{2n^2-5n-18 + \sqrt{4n^4-12n^3+33n^2+132n+36}}{6(n-2)}.$$ Plugging these eigenvalues in $P(x)$ we get, $P(\rho(Q_1))$ and $P(\rho(Q_2))$ are monotonically decreasing in $n$ with the limit 0. Since the two roots of $Q_G$ are negative, $P(x)>0$ for all $x\geq \min \{\rho(Q_1), \rho(Q
_2)\}$. Hence we conclude $\rho(Q_1)>\rho(G),$ $   \rho(Q_2)>\rho(G)$, and by eigenvalue interlacing (see Theorem \ref{thm:int_sr}), $\rho(G')>\rho(G).$

Suppose $G'$ has two adjacent vertices $u,v\in V'$ with $d(u)= d(v) = \frac{2n}{3}-1.$
Consider the partition of $G$ into two parts : $\{u,v\}, V\setminus\{u,v\}.$ The corresponding quotient matrix is 
$$Q = \begin{bmatrix}
    1&\frac{2n}{3}-2\\
    \frac{4n-12}{3(n-2)}&\frac{2n^2-8n+12}{3(n-2)}
\end{bmatrix}$$
and its characteristic polynomial is
$$g(x) = (x-1)\left(x-\frac{2n^2-8n+12}{3(n-2)}\right)-2\frac{(\frac{2n}{3}-2)^2}{n-2}$$ 
with the largest root $$\rho(Q) = \frac{2n^2-5n+6+\sqrt{4n^4-12n^3-63n^2+276n-252}}{6(n-2)}.$$
Plugging this eigenvalue in $P(x)$ we get, $P(\rho(Q))$ is monotonically decreasing in $n$ with the limit 0. Therefore $\rho(Q)>\rho(G)$ and by eigenvalue interlacing (see Theorem \ref{thm:int_sr}), $\rho(G')>\rho(G).$ The only remaining case to eliminate is when $G'$ has one vertex of degree $\frac{2n}{3}+1$, $n-4$ vertices of degree $\frac{2n}{3}$, and 3 vertices (say, $u,v,w$) of degree $\frac{2n}{3}-1.$ We proved above degree $\frac{2n}{3}-1$ vertices cannot be adjacent in a minimizer, so $u,v,w$ are pairwise non-adjacent in $G'$. Consider the partition of $G'$ into two parts : $\{u,v,w\}, V\setminus\{u,v, w\}.$ The corresponding quotient matrix is 
$$Q = \begin{bmatrix}
    0&\frac{2n}{3}-1\\
    \frac{2n-3}{n-3}&\frac{2n^2-12n+12}{3(n-3)}
\end{bmatrix}$$
and its characteristic polynomial is $$
        g(x) =x\left(x - \frac{2n^2-12n+12}{3(n-3)}\right) - 3\frac{(\frac{2n}{3}-1)^2}{n-1}$$ with the largest root
        $$\rho(Q) = \frac{n^2-6n+6+\sqrt{n^4-24n^2+63n-45}}{3(n-3)}.$$
 Plugging this eigenvalue in $P(x)$ we get, $P(\rho(Q))$ is monotonically increasing in $n$ with the $P(\rho(Q)) =0.662 $ for $n=6$. Therefore, $\rho(Q)>\rho(G)$ and by eigenvalue interlacing (see Theorem \ref{thm:int_sr}), $\rho(G')>\rho(G).$ This completes the proof.
\end{proof}

\begin{proof}[Proof of Proposition \ref{minimizer for e=n^2/3-1}]
Let $G' = (V',E')$ be a minimizer in $\mathcal{G}_{n,e}$. To prove $G'\cong G$, it is sufficient to show $G'$ contains an independent set of size $\frac{n}{3}$ as proving the remaining $\frac{2n}{3}$ vertices induce a subgraph isomorphic to $\hat{G}$ is similar to Proposition \ref{min for e=n^2/4-1}. Relabel the vertices of $G'$ so that the two vertices in $G'$ of degree $\frac{2n}{3}-1$ coincides with $\{x,y\}\in V.$ By proposition \ref{e=n^2/3-1}, we have $x\not\sim y$ in $G'$, hence $|N(x)\cap N(y)|\geq \frac{n}{3}.$ Let $S\subseteq N(x)\cap N(y)$ of cardinality $\frac{n}{3}. $ Relabel again the vertices of $G'$ other than $x,y$ such that the vertices in $S$ coincides with $V(G_1).$ If $S$ is not an independent set in $G'$, then by Theorem \ref{local switching} and Equation \ref{eigenvector entries}, we get $\rho(G')>\rho(G)$, which is a contradiction. Hence, $S$ is an independent set of size $\frac{n}{3}$ which proves $G' \cong G.$ The claim now follows from Equation \ref{char poly for e=n^2/3-1}.
\end{proof}

\section{Final Remarks}

In this paper, we proved that when $e = \frac{dn}{2}-1$, for $d\in\{2,\frac{n}{2}, \frac{2n}{3}, n-3, n-2, n-1\}$, the minimizer graph is a $d$-regular graph minus an edge. We believe the same is true for any value of $d$. We note that it is not necessarily true that when $e = \frac{dn}{2}+1$ for some $2\leq d\leq n-2$, a minimizer is obtained from a $d$-regular graph by adding an edge. For example,  when $n=7$ and $d=2$, the minimizers graphs in $\mathcal{G}_{n,\frac{dn}{2}+1}$ are in Figure \ref{bicyclic} and neither of them can be obtained by adding an edge to $C_7$.
\begin{figure}[H]
\centering
\begin{tikzpicture}[scale=0.8]\small
  \node[below] at (0,0) {$b$}; 
  \draw (0,0)--(2,0)--(1,1.5)--(0,0);
  \node[below] at (2,0) {$c$};
  \node[above] at (1,1.5) {$a$};
  \draw (2,0)--(4,0)--(6,0)--(8,0)--(7,1.5)--(6,0);
  \foreach \x in {(0,0), (1,1.5), (2,0),(4,0),(6,0),(8,0),(7,1.5)} {
  \draw[fill] \x circle[radius = 0.05cm];}
  \node[below] at (4,0) {$d$};
  \node[below] at (6,0) {$e$};
  \node[below] at (8,0) {$f$};
  \node[above] at (7,1.5) {$g$};
  \end{tikzpicture}
  \hspace{2cm}
  \begin{tikzpicture}[scale=0.8]\small
      \draw (-2,0)--(0,0)--(2,0)--(1,1)--(-1,1)--(-2,0);
      \draw(-2,0)--(-1,-1)--(1,-1)--(2,0);
      \foreach \x in {(-2,0),(0,0),(2,0),(1,1),(-1,1),(-1,-1),(1,-1)}
      {\draw[fill] \x circle[radius = 0.05cm];}
      \node[below] at (0,0) {$d$};
      \node[below] at (-1,-1) {$f$};
      \node[below] at (1,-1) {$g$};
      \node[above] at (1,1) {$b$};
      \node[above] at (-1,1) {$a$};
      \node[left] at (-2,0) {$c$};
      \node[right] at (2,0) {$e$};
  \end{tikzpicture}
  \caption{Bicyclic minimizers on $7$ vertices.}
  \label{bicyclic}
  \end{figure}
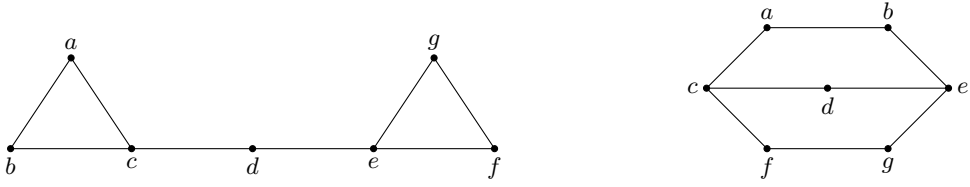

When $n$ does not divide $2e$, Hong's question asks whether the set of graphs in $\mathcal{G}_{n,e}$ with minimum spectral radius is contained in the set of graphs in $\mathcal{G}_{n,e}$ with $Ir(G)=1$ (which is the minimum possible). In general, we note that it is not true that with increase in $Ir(G)$, the spectral radius also increases. For example, consider the graphs in Figure \ref{contradiction}. The graph $G_1$ (on the left) has $\Delta(G_1)-\delta(G_1) = 4$ with $\rho(G_1) = 2.677$ and the graph $G_2$ (on the right) has $\Delta(G_2)-\delta(G_2)=3$ with $\rho(G_2) = 2.852$.


The following more general construction was suggested to us by Noga Alon\footnote{Personal communication to the authors in September 2023.}. Consider two natural numbers $t$ and $n$ such that $t\gg 1$ and $n$ $n\gg t$; (informally, $t$ is large and $n$ is much larger than $t$). Take a $t$-regular graph $H$ of order $n-t-2$ and the complete graph $K_{t+2}$ that are vertex disjoint. Construct the graph $G_1$ by adding an edge between $H$ and $K_{t+2}.$ The graph $G_1$ has size $e =\frac{(t+1)(t+2)}{2} + \frac{t(n-t-2)}{2} +1 $, $Ir(G_1) = 2$, and $\rho(G_1)\geq t+1$. Consider another graph $G_2$ with one vertex $v$ of degree $2t+4$ and all other vertices of degree $t$. Thus, $G_2\in\mathcal{G}_{n,e}$ and $Ir(G_2) = t+4>>2.$ Assume $G_2$ has high girth $2r+1$ which depends on $n$. Let $x$ be the principal eigenvector of $G_2$. By Claim 2.2 in \cite{Alon}, we have $||x||_\infty\leq \frac{1}{\sqrt{r}}.$ We denote by $N(i)$ the set of vertices at distance $i$ from vertex $v.$ Suppose $b^2 = \sum_{i\in N(0)\cup...\cup N(r)}x_i^2.$ Since $G_2$ has girth $2r$, it locally looks like a tree up to $N(r)$ from the vertex $v$.  Therefore,
\begin{equation*}
    \rho(G_2)=x^TA_{G_2}x \leq 2\sqrt{2t+3}b^2 +t\left(1-b^2+\frac{1}{r}\right)\leq t+\frac{t}{r}\leq t + o(1).
\end{equation*}
This shows although $Ir(G_1)<<Ir(G_2)$, but $\rho(G_1)>\rho(G_2)$ for large $n$.
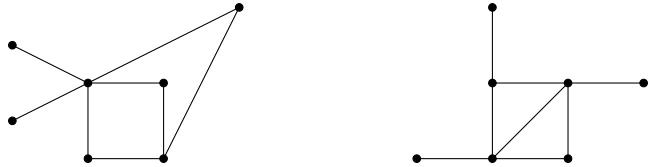
\begin{figure}[h]
\centering
\begin{tikzpicture}
  \foreach \x in {(0,0), (1,1), (2,2),(-1,0.5),(-1,1.5),(0,1),(1,0)} {
\draw[fill] \x circle[radius = 0.05cm];} 
\draw (0,0)--(1,0)--(1,1)--(0,1)--(0,0);  
\draw (0,1)--(2,2)--(1,0);
\draw (-1,0.5)--(0,1)--(-1,1.5);
\end{tikzpicture}
 \hspace{2cm}
\begin{tikzpicture}
 \foreach \x in {(0,0), (1,1), (2,1),(0,2),(-1,0),(0,1),(1,0)} {
\draw[fill] \x circle[radius = 0.05cm];} 
\draw (0,0)--(1,0)--(1,1)--(0,1)--(0,0);
\draw (-1,0)--(0,0)--(1,1)--(2,1);
\draw (0,2)--(0,1);
\end{tikzpicture}
\caption{Graphs $G_1$ (left) and $G_2$ (right) on 7 vertices with 8 edges.}
\label{contradiction}
\end{figure}

\section*{Acknowledgments}
We are grateful to the referees for several useful comments and suggestions. The research of the first author was partially supported by NSF grant DMS-2245556. The research of the second author was supported by the University of Delaware Graduate College through the Doctoral Fellowship for Excellence.
The research of the third author was partially supported by CNPq grant 405552/2023-8 and FAPERJ grant E-20/2022-284573. Part of this research was done during the Workshop on Spectral Graph Theory {\tt http://spectralgraphtheory.org/} that took place in Niteroi, Brazil in October 2023. We are grateful to the organizers and the participants to the workshop for creating a wonderful research atmosphere.

\end{document}